\documentclass[11pt]{article}
\usepackage [dvips]{graphics}
\usepackage[centertags]{amsmath}
\usepackage{amsfonts}
\usepackage{amssymb}
\usepackage{amsthm}
\usepackage{newlfont}
\usepackage{stmaryrd}
\usepackage{enumitem}
\usepackage{mathrsfs}

\theoremstyle{definition}
\newtheorem{definition}{Definition}[section]
\theoremstyle{plain}
\newtheorem{Theorem}[definition]{Theorem}
\newtheorem{lemma}[definition]{Lemma}
\newtheorem{corollary}[definition]{Corollary}
\newtheorem{proposition}[definition]{Proposition}

\newtheorem{exam}[definition]{Example}
\theoremstyle{remark}
\newtheorem{remark}[definition]{Remark}

\def\sqr#1#2{{\vcenter{\vbox{\hrule height.#2pt
              \hbox{\vrule width.#2pt height#1pt \kern#1pt \vrule
width.#2pt}
              \hrule height.#2pt}}}}
\def\dbA{{\mathbb{A}}}

\def\dbC{{\mathbb{C}}}

\def\3n{\negthinspace \negthinspace \negthinspace }
\def\2n{\negthinspace \negthinspace }
\def\1n{\negthinspace }

\def\cV{{\cal V}}
\def\cW{{\cal W}}

\def\no{\noindent}

\def\ms{\medskip}
\def\bs{\bigskip}

\def\dim{\hbox{\rm dim$\,$}}

\def\({\Big (}
\def\){\Big )}
\def\[{\Big[}
\def\]{\Big]}

\def\be{\begin{equation}}
\def\bel{\begin{equation}\label}
\def\ee{\end{equation}}
\def\bea{\begin{eqnarray}}
\def\eea{\end{eqnarray}}
\def\bt{\begin{theorem}}
\def\et{\end{theorem}}
\def\bc{\begin{corollary}}
\def\ec{\end{corollary}}
\def\bl{\begin{lemma}}
\def\el{\end{lemma}}
\def\bp{\begin{proposition}}
\def\ep{\end{proposition}}
\def\br{\begin{remark}}
\def\er{\end{remark}}
\def\ba{\begin{array}}
\def\ea{\end{array}}
\def\bd{\begin{definition}}
\def\ed{\end{definition}}

\makeatletter
   
   \@addtoreset{equation}{section}
\makeatother

\begin{document}

\title{\bf Commutativity of normal compact operators via projective spectrum \thanks{This work is supported by NSFC(No. 11471189)}}
\author{
Tong Mao\thanks{Taishan College, Shandong University, Jinan 250100, China. {\small\it
e-mail:} {\small\tt aaa1bbbb1ccc@sina.com}.}
 ~~~~~~
Yikun Qiao\thanks{Taishan College, Shandong University, Jinan 250100, China. {\small\it
e-mail:} {\small\tt qiaoyikun@hotmail.com}.}
~~~~~~
Penghui Wang\thanks{School of
Mathematics, Shandong University, Jinan 250100, China. {\small\it
e-mail:} {\small\tt phwang@sdu.edu.cn}. \ms}
}
\maketitle
\begin{abstract}
In this note we obtain commutativity criteria for   normal compact operators using the projective spectrum. We thus improve a corresponding result obtained by Chagouel, Stessin and Zhu in \cite{GST}.
\end{abstract}

\bs

\no{\bf 2000 Mathematics Subject Classification}.  Primary 47A13, 47A10.

\bs

\no{\bf Key Words}.  Normal operator, compact operator, projective spectrum, Agmon's Condition
\section{Introduction}

In \cite{Ya}, R. Yang introduced the concept of  projective spectrum.  For an $n$-tuple $\dbA=(A_1,\dots,A_n)$ of operators acting on a Hilbert space $H$,  the \emph{projective spectrum} of $\dbA$ is defined by
\begin{equation*}
\Sigma(\dbA)=\{(z_1,\dots,z_n)\in\dbC^n:z_1A_1+\dots+z_nA_n \text{ is noninvertible}\}.
\end{equation*}
Obviously, if $H$ is infinite-dimensional, and all of $A_i$'s are compact, then $\Sigma(\mathbb A)=\mathbb C^n$. To study the commutativity of normal compact operators, in \cite{GST}, the authors gave the following modified definition of projective spectrum
\begin{equation*}
\sigma(\dbA)=\{(z_1,\dots,z_n)\in\dbC^n:I+z_1A_1+\dots+z_nA_n \text{ is noninvertible}\},
\end{equation*}
and the point projective spectrum
\begin{equation*}
\sigma_p(\dbA)=\{(z_1,\dots,z_n)\in\dbC^n:\ker(I+z_1A_1+\dots+z_nA_n)\not=0\}.
\end{equation*}
By using the modified projective spectrum, I.Chagouel, M. Stessin and K. Zhu obtained the following theorem.
\begin{Theorem}[Chagouel, Stessin and Zhu, 2016]\label{CSZ}
Let $\mathbb A=(A_1,A_2,\dots,A_n)$ be an $n$- tuple of compact operators on a Hilbert space $H$. Suppose that
 \begin{itemize}
 \item[1)] each $A_i$ is self-adjoint and $\dim H=\infty$,
 \item[2)] each $A_i$ is normal and $\dim H<\infty$.
 \end{itemize}
Then the operators $A_1,\dots,A_n$ pairwise commute if and only if their projective spectrum $\sigma_p(\mathbb A)$ consists of countably many, locally finite, complex hyperplanes  in $\mathbb C^n$. Where, ``locally finite'' means that for each $z_0\in \mathbb C^n$, there is a neighborhood $U_0$ of $z_0$, such that $U_0\cap \sigma_p(\mathbb A)$ has finite branches.

\end{Theorem}

The paper \cite{GST} also pointed out that the theorem does not hold without a normality condition on the tuple. In the present paper, we will show that such a result is true for normal tuples under some mild conditions. As a particular case, we recover the cited result of I.Chagouel, M.Stessin and K.Zhu. In the following we shall use the notation from \cite{GST}. To state our result, we recall that an operator $A$ satisfies Agmon's condition\cite{Ag}, if there is a ray $\{\mathrm{Arg}\lambda=\theta\}$ such that $A$ has no eigenvalues on the ray.  With Agmon's condition, S. Seeley studied the complex powers of elliptic operators. Inspired by Agmon's condition, we introduce the following strengthening of Agmon's condition.

\begin{definition}\label{sagmon}
A normal compact operator $A$ is said to satisfy the strong Agmon condition, if there is an $\epsilon>0$ and  $\theta\in (0,2\pi)$ such that $A$ has no nonzero eigenvalues in $\{z: \theta-\epsilon<\mathrm{Arg} z<\theta+\epsilon\}$.
\end{definition}

The following theorem is the main result in the present note.

\begin{Theorem}\label{Main}
Let $\mathbb A=(A_1,A_2,\dots,A_n)$ be a tuple of  normal compact operators satisfying the strong Agmon condition, then the following conditions are equivalent
\begin{itemize}
\item[1)] $\mathbb A$ is commutative
\item[2)]$\sigma_p(\mathbb A)$ consists of countably many, locally finite, complex hyperplanes  in $\mathbb C^n$.
\end{itemize}
\end{Theorem}

Since self-adjoint compact operators and normal matrices satisfy the strong Agmon condition, Theorem \ref{CSZ} is a consequence of Theorem \ref{Main}. The result is proved as follows. At first we will need the following technical condition.
\vskip2mm
\noindent{\bf Condition A.} A normal compact operator $A$ is said to satisfy Condition A, if there is an $\epsilon>0$ such that the set
$
\bigcap_{\lambda\in\sigma_p(A)}\{z\in \mathbb C : \lvert1+\lambda z\rvert\geq\epsilon\}
$
is unbounded.
\vskip2mm
It will be shown that the strong Agmon condition implies Condition A. As in \cite{GST}, to get our main result, the key-point is to consider the case $n=2$. We will prove that if $A$ satisfies Condition A, $B$ is a normal compact operator, then  $[A,B]=0$ if and only if $\sigma_p(A,B)$ consists of countably many, locally finite, complex lines  in $\mathbb C^2$.

Compared to \cite{GST}, firstly, our proof is shorter and more elementary; Secondly, we do not need a stronger hypothesis for the case of normal operators. We conjecture that the result is true for normal compact operators without any extra condition.

\section{Proof of the main result}

In this section, we will prove our main theorem. At first, we will show that the strong Agmon condition implies Condition A.
\begin{lemma}\label{agmontowagmon}
	If a compact operator A satisfies the strong Agmon condition, then there exists  $0<\epsilon<1$ and a complex sequence $\{z_n\}_{n\in \mathbb N}$ such that
		\begin{equation*}
			\lim_{n\rightarrow\infty}z_n=\infty
		\end{equation*}
and for every $\lambda\in\sigma_p(A)$ and $n\in \mathbb N$
		\begin{equation*}
			\lvert 1+\lambda z_n\rvert\geq \epsilon
		\end{equation*}
\end{lemma}

\begin{proof}
By Definition \ref{sagmon}, there exists $0\leq\theta<2\pi$ and $0<\delta<\pi$ such that $$\sigma_p(e^{i\theta}A) \backslash \{0\}\subseteq\{z\in \mathbb C : 0\leq\mathrm{Arg}(z)<\pi-\delta\quad or\quad \pi+\delta<\mathrm{Arg}(z)<2\pi\}.$$
Take $0<\epsilon<\sin\delta$, $z_n=e^{i\theta}n$, then $\lim_{n\rightarrow\infty}z_n=\infty$. Now, for any $\lambda\in \sigma_p(A)$
\begin{equation*}
e^{i\theta}\lambda\in\sigma_p(e^{i\theta}A)\subseteq\{z\in \mathbb C : 0\leq\mathrm{Arg}(z)<\pi-\delta\quad or\quad \pi+\delta<\mathrm{Arg}(z)<2\pi\}\cup\{0\}
\end{equation*}
Obviously, if $\lambda=0$
\begin{equation*}
\lvert1+\lambda z_n\rvert=1\geq \epsilon
\end{equation*}

If $\lambda \ne 0$, then $-\frac{1}{e^{i\theta}\lambda}\in\{z\in \mathbb C : \delta<\mathrm{Arg}z<2\pi-\delta\}$, since $\mathrm{Arg}(e^{i\theta}\lambda)=\pi-\mathrm{Arg}(-\frac{1}{e^{i\theta}\lambda})$. The distance between $-\frac{1}{e^{i\theta}\lambda}$ and the positive $x\text-axis$ is
\begin{equation*}
\inf_{x>0}\big\lvert x- (-\frac{1}{e^{i\theta}\lambda})\big\rvert\geq \frac{\sin \delta}{\lvert \lambda\rvert}
\end{equation*}
then
\begin{equation*}
\lvert1+\lambda z_n\rvert=\lvert \lambda\rvert\cdot\big\lvert z_n-(-\frac{1}{\lambda})\big\rvert=\lvert \lambda\rvert\cdot\big\lvert n-(-\frac{1}{e^{i\theta}\lambda})\big\rvert\geq\sin\delta\geq \epsilon
\end{equation*}

\end{proof}

\begin{lemma}\label{lemma3.4}
	For compact operators $A$ and $B$, suppose $A$ is normal and satisfies Condition A.
If $\mu\ne0$ is a complex number such that the complex line $\{(z,w)\in\dbC^2:\mu w+1=0\}$ is contained in $\sigma_p(A,B)$, and $\lvert \mu \rvert=\lVert B\rVert$, then there exists a unit vector $x$ such that
\begin{equation}\label{35}
		Ax=0\quad \text{and}\quad Bx=\mu x.
\end{equation}
\end{lemma}

\begin{proof}
Write
		\begin{equation*}
			A=\sum_j \lambda_j e_j\otimes e_j,
		\end{equation*}
where $\{e_j\}$ is an orthonormal sequence of eigenvectors of $A$ with corresponding eigenvalues $\lambda_j$. Since $A$ satisfies Condition A, there exists  $0<\epsilon<1$ and a complex sequence $\{z_n\}_{n\in \mathbb N}$ such that
		\begin{equation*}
			\lim_{n\rightarrow\infty}z_n=\infty,
		\end{equation*}
and for every $j\in \mathbb N$ and $n\in \mathbb N$
		\begin{equation*}
			\lvert 1+\lambda_j z_n\rvert\geq \epsilon.
		\end{equation*}
Because the complex line $\mu w+1=0$ is contained in $\sigma_p(A,B)$, for every $z\in \mathbb C$, $I+zA-\frac{1}{\mu}B$ has nontrivial kernel. There exists a unit vector $v_n$, such that
\begin{equation}																			\label{39}
		\(I+z_nA-\frac{1}{\mu}B\)v_n=0.
\end{equation}

Since the unit ball of a Hilbert space is weakly compact,  there exists a subsequence $\{v_{n_k}\}$ of $\{v_n\}$  which converges weakly to some vector $v_0\in H$. Since $A, B$ are compact, we have
\begin{eqnarray}
	\lim\limits_{k\to\infty}Av_{n_k}=Av_0,\quad\text{and}\quad
	\lim\limits_{k\to\infty}Bv_{n_k}= Bv_0.\label{41}
\end{eqnarray}

Let $P_0$ be the orthogonal projection onto ker$A$. Now, we claim that $v_0\not=0$. To see this, we argue by contradiction.
Assume $v_0=0$, then
\begin{equation}\label{47}
\lim_{k\to\infty}(I+z_{n_k}A)v_{n_k}=\lim_{k\to\infty}\frac{1}{\mu}Bv_{n_k}=\frac{1}{\mu}Bv_0=0
\end{equation}

In the basis $\{e_j\}_j$
\begin{equation*}
(I-P_0)(I+z_{n_k}A)v_{{n_k}}=\sum_j(1+\lambda_j z_{n_k})\langle v_{n_k},e_j\rangle e_j,
\end{equation*}
which tends to 0, that is
\begin{equation}\label{49}
\lim_{k\to\infty}\sum_j\lvert1+\lambda_j z_{n_k}\rvert^2\lvert\langle v_{n_k},e_j\rangle\rvert^2=0.
\end{equation}

Recall that $\lvert1+\lambda_j z_n\rvert\geq \epsilon$. Combining with (\ref{47})
\begin{eqnarray*}
\sum_j\lvert1+\lambda_j z_{n_k}\rvert^2\lvert\langle v_{n_k},e_j\rangle\rvert^2
&\geq& \epsilon^2\sum_j\lvert\langle v_{n_k},e_j\rangle\rvert^2\\
&=&\epsilon^2(\lVert v_{n_k}\rVert^2-\lVert P_0v_{n_k}\rVert^2)\\
&=&\epsilon^2(1-\lVert P_0(I+z_{n_k}A)v_{n_k} \rVert^2)\\
&\rightarrow&\epsilon^2
\end{eqnarray*}
which contradicts to (\ref{49}).
By (\ref{39}) and (\ref{41})
\begin{equation*}
	Av_0=\lim\limits_{k\to\infty}Av_{n_k}=\lim\limits_{k\to\infty}\frac{1}{z_{n_k}}(-I+\frac{1}{\mu}B)v_{n_k}=0,
\end{equation*}
by which $v_0\in \mathrm {ker}A $. Recall that $P_0$ is the orthogonal projection onto ker$A$, then
\begin{eqnarray*}
v_0&=&P_0v_0\\&=&w-\lim\limits_{k\to\infty}P_0v_{n_k}\\&=&w-\lim\limits_{k\to\infty}P_0(-z_{n_k}A+\frac{1}{\mu}B)v_{n_k}\\&=&w-\lim\limits_{k\to\infty}\frac{1}{\mu}P_0Bv_{n_k}\\&=&\frac{1}{\mu} P_0Bv_0.
\end{eqnarray*}

Since $\lvert\mu\rvert=\lVert B\rVert$, by Pythagorean theorem
\begin{eqnarray*}
\lVert(I-P_0)Bv_0\rVert^2&=&\lVert Bv_0\rVert^2-\lVert P_0Bv_0\rVert^2\\&\leq&\lVert B\rVert^2\lVert v_0\rVert^2-\lVert\mu v_0\rVert^2=0
\end{eqnarray*}
that is
\begin{equation}
Bv_0=P_0Bv_0=\mu v_0
\end{equation}
which shows that $v_0$ is a common eigenvector of $A$ and $B$. By normalizing $v_0$, we get the unit vector $x$ satisfying (\ref{35}).

\end{proof}

We also need the following lemma.

	\begin{lemma}\label{lemma2.1}
	For compact operators $A$ and $B$, suppose $A$ is normal and $(\lambda,\mu)\ne(0,0)$ are complex numbers such that the complex line $\{(z,w)\in\dbC^2:\lambda z+\mu w+1=0\}$ is contained in $\sigma_p(A,B)$,  and $\lambda$ is an isolated eigenvalue of $A$. Then there exists a unit vector $x$, such that
		\begin{equation}\label{4}
			Ax=\lambda x\quad and \quad \mu=\langle Bx,x\rangle
		\end{equation}
	\end{lemma}

	\begin{proof}
		We can choose a disc $D=D(\lambda,\delta)$ containing $\lambda$ for a small $\delta> 0$ such that:
		\begin{enumerate}[itemindent=2em]
			\item $0\notin D$ if $\lambda\neq 0$
			\item $D\cap \sigma_p(A)=\{\lambda\}$,
			\item $uI-A$ is invertible for $u\in \partial D$.
		\end{enumerate}
	Define:
		\begin{equation*}
			A_\epsilon:=A+\epsilon B\qquad \lambda_\epsilon:=\lambda+\epsilon \mu
		\end{equation*}
Take $\sigma>0$ small enough such that for $0<\lvert\epsilon\rvert<\sigma$, $uI-A_\epsilon$ is invertible for $u\in\partial D$, $\lambda_\epsilon\in D$ and $\lambda_\epsilon\neq 0$. Since $(-\frac{1}{\lambda_\epsilon},-\frac{\epsilon}{\lambda_\epsilon})\in \sigma_p(A,B)$ and
		\begin{equation*}
			\lambda_\epsilon I-A_\epsilon =\lambda_\epsilon(I-\frac{1}{\lambda_\epsilon}A-\frac{\epsilon}{\lambda_\epsilon}B),
		\end{equation*}
we have $\lambda_\epsilon$ is an eigenvalue of $A_\epsilon$.
 For any fixed $\epsilon>0$ small enough, 	take a unit $v_\epsilon$ such that
		\begin{equation*}
			(A_\epsilon-\lambda_\epsilon I )v_\epsilon=0.
		\end{equation*}
Consider the Riesz projections\cite{Con}:
		\begin{equation*}
			P_\epsilon =\frac{1}{2\pi i}\int_{\partial D} (uI-A_\epsilon)^{-1} \mathrm d u
		\end{equation*}
and
		\begin{equation}\label{9}
			P_0=\frac{1}{2\pi i}\int_{\partial D} (uI-A)^{-1} \mathrm d u,
		\end{equation}
then $P_\epsilon\rightarrow P_0$ as $\epsilon \rightarrow 0$. Obviously $P_\epsilon v_\epsilon=v_\epsilon$.
Rewrite $P_\epsilon$ as
		\begin{eqnarray*}
				P_\epsilon &=& \frac{1}{2\pi i}\int_{\partial D} (uI-A_\epsilon)^{-1} \mathrm d u\\
								&=&\sum_{r=0}^\infty\frac{1}{2\pi i}\int_{\partial D} \epsilon ^r((uI-A)^{-1}B)^r (uI-A)^{-1} \mathrm d u\\
								&=&\frac{1}{2\pi i}\int_{\partial D} (uI-A)^{-1} \mathrm d u \\
								&& + \frac{1}{2\pi i}\epsilon \int_{\partial D}(uI-A)^{-1}B (uI-A)^{-1}\mathrm d u+ O(\epsilon^2)\\
								&=&P_0+\epsilon \tilde P+O(\epsilon^2),
		\end{eqnarray*}
where $\tilde P=\frac{1}{2\pi i}\int_{\partial D}(uI-A)^{-1}B (uI-A)^{-1}\mathrm d u$. Accordingly $(A_\epsilon-\lambda_\epsilon I)P_\epsilon$ can be written as
		\begin{eqnarray}
							(A_\epsilon-\lambda_\epsilon)P_\epsilon &=&\big(A-\lambda I+\epsilon(B-\mu I)\big)\big(P_0+\epsilon \tilde P+O(\epsilon^2)\big)\nonumber\\
																			&=&(A-\lambda I)P_0+\epsilon\big((A-\lambda I)\tilde P+(B-\mu I)P_0\big)+O(\epsilon^2)\label{12}
		\end{eqnarray}
Please note that
		\begin{equation*}
			(A-\lambda I)P_0=P_0(A-\lambda I)=0.
		\end{equation*}
Multiplying $P_0$ to the left of (\ref{12})
		\begin{equation}\label{16}
			P_0(A_\epsilon-\lambda_\epsilon I)P_\epsilon=\epsilon P_0(B-\mu I)P_0+O(\epsilon^2)
		\end{equation}
Recall that $v_\epsilon$ is a unit eigenvector, together with (\ref{16})
		\begin{equation}\label{17}
			P_0(B-\mu I)P_0v_\epsilon=O(\epsilon)
		\end{equation}

\vskip1.5mm	If $\lambda\neq0$. Then since $A$ is compact, the range $\mathrm{Ran}P_0$ is of finite dimension. Thus we can choose a converging subsequence of $\{P_0v_\epsilon\}$ with the limit $v_0$. In (\ref{17}), let $\epsilon\rightarrow 0$ in the subsequence
		\begin{equation}\label{18}
			P_0(B-\mu I)v_0=0.
		\end{equation}
	We have $\lVert v_0\rVert=1$ because
		\begin{equation*}
			1\geq\lVert v_0\rVert\geq \lVert P_\epsilon v_\epsilon\rVert-\lVert P_\epsilon v_\epsilon- P_0v_\epsilon\rVert -\lVert v_0-P_0 v_\epsilon\rVert.
		\end{equation*}

\vskip1.5mm	If $\lambda=0$, then $\mu\neq0$. Consider $\tilde{B}=P_0(B-\mu I)P_0$ and an operator on $\mathrm{Ran} P_0$. Then $\tilde{B}$ has nontrivial kernel.   Otherwise suppose it were injective. Since $P_0BP_0$ is compact, by Riesz-Schaulder theory, $\tilde{B}$ is invertible. Therefore there exists $d>0$ such that
		\begin{equation*}
			 \lVert P_0(B-\mu I)P_0v\rVert\geq d\lVert P_0v\rVert,\quad \textrm{for all } v\in H,
		\end{equation*}
which contradicts to (\ref{17}).

In summary, there is a unit vector $v_0$ such that (\ref{18}) holds whether $\lambda=0$ or not. 	Let $x=v_0$, we have (\ref{4}).
\end{proof}

From the above technical lemma, we have
\begin{corollary}\label{cor2.4}
Let $A$ and $B$ be normal compact operators such that $A$ satisfies Condition A. If $\sigma_p(A,B)$ consists of complex lines, then $A$ and $B$ have a common eigenvector.
\end{corollary}
\begin	{proof}
Choose $\mu$ to be the eigenvalue of $B$ with maximal norm, that is $|\mu|=\lVert B\rVert$. The case $\mu=0$ is trivial, so suppose $\mu\ne 0$. The point $(0,-\frac{1}{\mu})$ is contained in $\sigma_p(A,B)$. By the assumption on $\sigma_p(A,B)$, there is  a complex line $\lambda z+\mu w+1=0$ in $\sigma_{p}(A,B)$ containing $(0,-{1\over\mu})$.

If $\lambda\ne0$, then $(-\frac{1}{\lambda},0)$ is contained in $\sigma_p(A,B)$, which indicates $\lambda$ is a nonzero eigenvalue of $A$. By Lemma \ref{lemma2.1} we have the desired result.

If $\lambda=0$, the corollary comes from Lemma \ref{lemma3.4}.
\end{proof}

Suppose $A$ and $B$ satisfy the conditions in Corollary \ref{cor2.4}. Define two sets of subspaces of $H$
\begin{eqnarray*}
\cV&=&\{V\subseteq H:A(V)\subseteq V,\,B(V)\subseteq V\}\\
\cW&=&\{W\in\cV:AB=BA\text{ on }W\}
\end{eqnarray*}
we have $0\in\cW$, and $A$ and $B$ commute if and only if $H\in\cW$.

By Zorn's lemma, $\cW$ has a maximal element $W$ with respect to inclusion, and we argue by contradiction to show $W=H$. $W$ is closed since $\overline W\in\cW$. Assume that $W\subsetneq H$, that is $W^\perp\neq0$. If there exists a common eigenvector of $A$ and $B$ in $W^\perp$, let $W'$ be the subspace generated by the vector, then $0\neq W'\subseteq W^\perp$ such that $W'\in\cW$, then $W\oplus W'\in\cW$, which contradicts the maximality of $W$.

$W^\perp\in\cV$ because $A$ and $B$ are normal. Denote the restricted operators on the Hilbert space $W^\perp$ by $A'$ and $B'$. We only need to show that $A'$ and $B'$ have a common eigenvector. This is done if the operators $A'$ and $B'$ on the Hilbert space $W^\perp$ satisfy the conditions in Corollary \ref{cor2.4}, which is assured by the following proposition

\begin{proposition}\label{prop}
Let $A$ and $B$ be normal compact operators such that $\sigma_p(A,B)$ consists of countably many, locally finite, complex lines in $\dbC^2$. If $W$ is a closed invariant subspace of both $A$ and $B$, then the restricted operators on $W$ have the same property as $A$ and $B$, that is $A\rvert_W$ and $B\rvert_W$ are normal compact operators over the Hilbert space $W$ such that $\sigma_p(A\rvert_W,B\rvert_W)$ consists of countably many, locally finite, complex lines.
\end{proposition}
\begin{proof}
This can be concluded from the proof of Theorem 11 of \cite{GST}.
\end{proof}

The reason that the commutativity of $A$ and $B$ implies that  $\sigma_p (A,B)$ consists of countably many, locally finite, complex lines in $\mathbb C ^2$ is trivial, since $A$ and $B$ are diagonalized by an orthonormal basis, see the proof of Theorem 11 in \cite{GST} for detail. We have our main result.
\begin{Theorem}\label{thm4}
If $A$ and $B$ are normal and compact, and $A$ satisfies Condition A, then the following conditions are equivalent:
			\begin{enumerate}
				\item $A$, $B$ are commutative,
				\item $\sigma_p (A,B)$ consists of countably many, locally finite, complex lines in $\mathbb C ^2$.
			\end{enumerate}
\end{Theorem}

Because self-adjoint operators and finite rank operators satisfy the strong Agmon condition automatically, by Lemma \ref{agmontowagmon} and Theorem \ref{thm4}, we have

\begin{corollary}\label{cor}
Let $A$ and $B$ are normal compact operators. Suppose $A$ is self-adjoint or finite rank. Then the followings are equivalent:
			\begin{enumerate}
				\item $A$, $B$ are commutative,
				\item $\sigma_p (A,B)$ consists of countably many, locally finite, complex lines in $\mathbb C ^2$.	
			\end{enumerate}
\end{corollary}
Obviously, if both $A$ and $B$ are finite rank, then the commutativity of $A$ and $B$ is equivalent to the finite dimensional case, and Corollary \ref{cor} recover Theorem \ref{CSZ}. Next, we give an example which shows that there is a normal compact operator that does not satisfy Condition A.

\begin{exam}
Let $H$ be Hilbert space with an orthonormal basis $\{e_{n,i}: n\in \mathbb N; 1\leq i\leq 2^n\}$. Set $\omega_{n,i}$ be the ith root of $x^{2^n}=1$. Let $\nu_{n}=\sum\limits_{j=1}^n{1\over j}$. Then $$\lambda_{n,i}={1\over \nu_n \omega_{n,i}}\to 0.$$ It is easy to verify that the operator $A=\sum\limits_{n,j}\lambda_{n,j}e_{n,j}\otimes e_{n,j}$ does not satisfy Condition A.
\end{exam}

\qed

\vskip2.5mm\noindent {\bf Acknowledgements.} The third author would like to thank Prof. Kehe Zhu for several discussions on questions in this note. The authors would like to thank the referee(s) sincerely for valuable suggestions which make the paper more readable.


\end{document}